\documentclass[a4paper,11pt]{article} 
\usepackage{float}
\usepackage{caption}
\usepackage{booktabs} 
\usepackage{tikz-qtree}
\usetikzlibrary{arrows.meta}
\usepackage[a4paper, margin=.9in]{geometry}
\usepackage{amsmath,amsthm,amsfonts,amssymb,array}
\usepackage[square,sort,comma,numbers]{natbib}
\usepackage{subeqnarray}
\usepackage{blkarray}
\usepackage[utf8]{inputenc}
\usepackage[english]{babel}
\usepackage{mathabx}
\usepackage{mathrsfs}
\usepackage{colortbl}
\usepackage{graphicx,epsfig}
\newcounter{myeqno}
\usepackage{color}
\usepackage{tcolorbox}
\usepackage{ upgreek }
\usepackage{relsize}
\usepackage{multirow}
\usepackage{MnSymbol}
\definecolor{shadecolor}{gray}{0.75}

\theoremstyle{definition}
\newtheorem{definition}{Definition}[section]

\newtheorem{proposition}{Proposition}[section]

\newtheorem{theorem}{Theorem}[section]

\newtheorem{corollary}{Corollary}[section]

\newtheorem{lemma}{Lemma}[section]

\newtheorem{remark}{Remark}[section]

\usepackage[misc,geometry]{ifsym}
\date{}
\usepackage[colorlinks=true]{hyperref}
\hypersetup{urlcolor=blue, citecolor=blue}
\newlength{\defbaselineskip}
\newcommand{\setlinespacing}[1]%
{\setlength{\baselineskip}{#1 \defbaselineskip}}

\begin{document}

		\title{\textbf{On the resistance regular graphs}}

\author{Haritha T$^1$\footnote{harithathottungal96@gmail.com},  Chithra A. V$^1$\footnote{chithra@nitc.ac.in}
	\\ \\ \small 
 1 Department of Mathematics, National Institute of Technology Calicut,\\\small Calicut-673 601, Kerala, India\\ \small}

\maketitle	
\begin{abstract}
For a connected graph $G$, its resistance distance matrix is denoted by $R(G)$. A graph is called resistance regular if all the row (or column) sums of $R(G)$ are equal. We provide a necessary and sufficient condition for a simple connected graph to be resistance regular. Additionally, we establish sharp bounds for the resistance spectral radius and present various bounds for the resistance energy of $G$. Furthermore, we compute the resistance spectrum and resistance energy of some resistance regular graphs.
\end{abstract}

\section{Introduction}
All graphs considered in this paper are simple, connected, and undirected. Let $G$ be a graph with vertices $v_1,\ldots, v_n$ and edges $e_1,\ldots, e_m.$  The \textit{complete graph} on $n$ vertices, \textit{cycle}  having $n$ vertices, and \textit{complete multipartite graph} on $q$ parts ($2\leq q \leq n-1$) are denoted by $K_n$, $C_n$, and $K_{n_1, \ldots, n_q}$ respectively, and when $n_1=\cdots=n_q=2, \; K_{2,\ldots, 2}\;(n_1+\cdots+n_q=n)$ is said to be the cocktail party graph ($CP(n)$). A \textit{walk-regular graph} is a graph for which the number of closed walks of length $t$, where $t\geq 2$, at a vertex $v$ is independent of the choice of $v.$ Denote the identity matrix by $I$ and the all-one square matrix by $J$, both of appropriate orders. The adjacency matrix $A(G)$ of a graph $G$ is a square symmetric matrix whose entries equal to one if the corresponding vertices are adjacent and zero otherwise. The set of all neighbors of a vertex $v_i\in V(G)$ is called its \textit{neighborhood} $N(v_i)$ in $G.$ Let the degree of a vertex $v_i$ in a graph $G$ be denoted by $\mathrm{deg}_G(v_i)$. The matrix $D(G)$ is a diagonal matrix whose diagonal entries are the vertex degrees, with the $i^{\text{th}}$ diagonal entry equal to $\mathrm{deg}_G(v_i)$. The \textit{Laplacian matrix} of $G$ is defined by $L(G)=L=(l_{ij})_{n\times n}= D(G)-A(G).$ Denote the eigenvalues of $L(G)$ by $\gamma_1, \ldots, \gamma_{n-1}, \gamma_n=0$, arranged in non-increasing order. For $1\leq i\leq n$, let $S_i=(s_{i1},s_{i2},\ldots, s_{in})^T$ be the real-valued, normalized, and mutually orthogonal eigenvectors corresponding to the eigenvalues $\gamma_i$ of $L(G).$ If $S=(s_{ij})_{n\times n}$, then $SS^T=S^TS=I.$
 The \textit{resistance distance} \cite{klein1993resistance}, a novel distance function, was introduced by Klein et al. in $1993$ as a new concept of distance in graphs.
For an $m\times n$ matrix $P$, the \textit{Moore-Penrose inverse} $P^\dagger$ \cite{ben2003generalized} of $P$ is the unique matrix 
$W$ such that $PWP= P$, $WPW= W$, $(PW)^T= PW$ and $(WP)^T= WP$.
     For an $m\times n$ matrix $P$, the matrix $M$ of order $n\times m$ is said to be a \textit{$\{ 1\}$-inverse}, $P^{(1)}$  of $P$  \cite{ben2003generalized} if $PMP = P$.

The Moore-Penrose inverse and a $\{1\}$-inverse of $L$ of the underlying graph $G$ are used to calculate the resistance distance $r_{ij}$ (or $r_{ij}(G)$) \cite{bapat2010graphs} between two vertices $v_i$ and $v_j$. The relation is as follows:

 \begin{equation*}
     \begin{aligned}
         r_{ij} &= l^{(1)}_{ii}+ l^{(1)}_{jj}-l^{(1)}_{ij}-l^{(1)}_{ji}=l_{ii}^{\dagger}+l_{jj}^{\dagger}-2l_{ij}^{\dagger}.
     \end{aligned}
 \end{equation*}

\noindent The matrix $R(G)=R= (r_{ij})_{n\times n}$ is called the resistance distance matrix of $G$. The eigenvalues
 of $R(G)$ are said to be the resistance eigenvalues ($R$-eigenvalues) of $G$. Let
 $\rho_1(G),\rho_2(G),\ldots,\rho_t(G)$ be the distinct $R$-eigenvalues of $G$ with multiplicities $s_1,s_2,\ldots,s_t$. Then the spectrum of $R(G)$ is denoted by $Spec_{R}(G)= \begin{pmatrix}
        \rho_1(G)&\rho_2(G)&\ldots&\rho_t(G)\\
        s_1&s_2&\ldots &s_t
    \end{pmatrix}$, and indexed such that $\rho_1(G)\geq \cdots \geq \rho_t(G)$.\\ The resistance spectral
 radius ($\rho_1(G)$) is defined as the maximum of absolute values of $R$-eigenvalues. \\
 The resistance energy ($R$-energy), $E_R(G)$ \cite{MR2951686} of $G$ is defined as $$E_R(G)= \sum_{i=1}^{n}|\rho_i(G)|.$$\\
 \noindent The \textit{Kirchhoff index} $\mathcal{K}f(G)$ of $G$, also known as the total resistance of a network is defined as,
$$\mathcal{K}f(G)= \sum_{i<j} r_{ij}.$$
Let the row sums of the resistance distance matrix be the resistance degrees of $G$, denoted as, $R_i$ for $i\in \{1,\ldots, n\}.$ A graph $G$ is said to be $k$-resistance regular if $R_i=k$ for all $i\in \{1,\ldots, n\}.$ The following example gives a $4$-resistance regular graph on $9$ vertices.
\begin{figure}[htbp]
\centering
 \begin{tikzpicture}[scale=0.6,inner sep=1.3pt]
\draw (0,2) node(8) [circle,draw,fill] {};
     \draw (2,0) node(4) [circle,draw,fill] {};
       \draw (-2,0) node(3) [circle,draw,fill] {};     \draw (2,2) node(9) [circle,draw,fill] {};
       \draw (-2,2) node(7) [circle,draw,fill] {};
       \draw (-4,2) node(2) [circle,draw,fill] {};
       \draw (4,2) node(5) [circle,draw,fill] {};
       \draw (-2,4) node(1) [circle,draw,fill] {};
       \draw (2,4) node(6) [circle,draw,fill] {};
 \draw [-] (6) to (1) to (2) to (3) to (4) to (5) to (6) to (8) to (1) to (9) to (5);
 \draw[-] (2) to (7) to (6) to (7) to (3) to (8) to (4) to (9);
 \draw [-,bend right] (7) to (5);
 \draw [-,bend right] (2) to (9);

\end{tikzpicture}
\caption{$4$-resistance regular graph on $9$ vertices.}
 \end{figure}
Denote the second resistance degree as $T_i= \sum_{j=1}^{n}r_{ij}R_j$, and the average resistance degree as $\overline{R}_i= \frac{T_i}{R_i}.$  A graph $G$ is said to be pseudo $k$-resistance regular if $\frac{T_i} {R_i}= k$ for all $i\in \{1,\ldots, n\}.$
\par The importance of resistance distance in graphs extends to combinatorial matrix theory as well as spectral graph theory \cite{bapat2010graphs,chen2007resistance,bapat2003simple}. See \cite{evans2022algorithmic} for an overview of techniques for determining resistance distance in graphs. In \cite{sarma2023study,zhou2016resistance}, the authors studied resistance distance matrix and resistance regular graphs. \\
This paper is organized as follows: In Section $2$, we present some known results. Section $3$ provides a necessary and sufficient condition for a graph to be resistance regular. Generally, calculating the resistance energy of graphs is challenging. Section $4$ determines the resistance spectrum and energy of the graphs formed using unary and binary operations and presents several bounds for the resistance energy. In addition, we obtain some sharp bounds for the resistance spectral radius of $G$ and derive an upper bound for the Kirchhoff index.
\section{Preliminaries}

This section presents some known results that will be used in the subsequent sections.
\begin{lemma}\cite{ressub}\label{L}
Let $G$ be a connected graph and let $L(G)= \begin{bmatrix}L_1& L_2\\ L_{2}^{T}& L_3\end{bmatrix}$. If each column vector of $L_{2}^{T}$ is $-J$ or a zero vector, then a $\{1\}$-inverse of $L(G)$ is $L^{(1)}= \begin{bmatrix}L_{1}^{-1}& 0\\ 0& S^{\dagger}\end{bmatrix}$, where $S= L_{3}-L_{2}^{T}L_{1}^{-1}L_{2}.$
\end{lemma}
\begin{lemma}{\cite{1inv}}\label{pinv}
   Let $\Gamma$ be a connected graph and let $L(\Gamma)= \begin{bmatrix}
        L_1& L_2\\L_2^T& L_3
    \end{bmatrix}$.
If $L_3$ is non-singular, then a $\{1\}$-inverse of $L(\Gamma)$ is
$$L^{(1)}= \begin{bmatrix}
    M^{\dagger}& -M^{\dagger}L_2L_3^{-1}\\
    -L_3^{-1}L_2^TM^{\dagger}& L_3^{-1}+L_3^{-1}L_2^TM^{\dagger}L_2L_3^{-1}
\end{bmatrix},$$
where $M= L_1-L_2L_3^{-1}L_2^T.$

\end{lemma}
\begin{lemma}\label{rescycle}\cite{flower}
    Let $C_n$ be a cycle of order $n.$ Then the resistance distance between any vertices $v_i, v_j\in V(C_n)$ is given by the formula:
    $$r_{ij}= \frac{(n-d_{ij})d_{ij}}{n}.$$
\end{lemma}
\begin{definition}\cite{indulal2006pair}
    Let $G$ be a graph with $V(G) = \{v_1, v_2,\ldots ,v_n\}$. Make another copy of $G$ with vertices $\{v_{1}', v_{2}',\ldots ,v_{n}'\}$ in which $v_{i}'$ corresponds to $v_i$ for each $i$ in such a way that for each $i, \; v_{i}'$ is adjacent to every vertices in $N(v_i)$. The resultant graph is called the \index{double graph}\textit{double graph} $D_2G$ of $G$.

\end{definition}

\begin{theorem}\cite{huang2016resistance}\label{resdouble}
    Let $G$ be a connected graph with degree sequence $(\mathrm{deg}_G(u_1), \ldots, \mathrm{deg}_G(u_n))$. And let $\{u_1,\ldots,u_n,u_{1}',\ldots, u_{n}'\}$ be the vertex set of the double graph $D_2G$ of $G.$ Then the resistance distance between two vertices $v_i, v_j$ in $V(D_2G)$ is
   $$r_{ij}(D_2G)=\begin{cases}\frac{1}{\mathrm{deg}_{G}(u_i)}, &\text{if $v_i=u_i$ and $v_j=u_{i}'$,}\\
     0, &\text{if $v_i= v_j$,}\\\frac{1}{4}\left(r_{ij}(G)+\frac{1}{\mathrm{deg}_{G}(u_i)}+\frac{1}{\mathrm{deg}_{G}(u_j)}\right), &\text{otherwise.}\end{cases}$$
\end{theorem}
\begin{definition}\cite{MR2977757}
   For $i\in \{1,2\}$, let $G_i$ be a graph with vertex set $V(G_i)$ and edge set $E(G_i).$ Then the \index{lexicographic product}\textit{lexicographic product} $G_1[G_2]$ of $G_1$ and $G_2$ is the graph with $V(G_1[G_2])= V(G_1)\times V(G_2)$ in which two vertices ($v_1, u_1$) and ($v_2, u_2$) are adjacent if $v_1v_2\in  E(G_1)$ or $v_1 = v_2$ and
$u_1u_2\in E(G_1)$.  
\end{definition}
\begin{proposition}\cite{yang2014resistance}\label{reslexico}
    Let $G$ be a connected graph with degree sequence $(\mathrm{deg}_G(v_1), \ldots, \mathrm{deg}_G(v_n))$. And let $\{u_1,u_2\}$ be the vertex set of $K_2.$ Then the resistance distance between two vertices $w_i=(v_i, u_i), w_j= (v_j, u_j)$ in the lexicographic product $G[K_2]$ is
    $$r_{ij}(G[K_2])=\begin{cases}\frac{1}{\mathrm{deg}_{G}(v_i)+1}, &\text{if $v_i=v_j$ and $u_i\neq u_j$,}\\
     \frac{r_{ij}(G)}{4}+\frac{1}{4\mathrm{deg}_G(v_i)+4}+\frac{1}{4\mathrm{deg}_G(v_j)+4}, &\text{otherwise.}\end{cases}$$
\end{proposition}

\begin{lemma}\cite{gutman2004generalized}\label{lldagger}
    Let $G$ be a connected graph with $n$ vertices and $L$ be its Laplacian matrix, then $LL^{\dagger}=L^{\dagger}L=I-\frac{1}{n}J.$
\end{lemma}
\begin{theorem}\cite{xiao2003resistance}\label{evectorformX}
    Let $G$ be a connected graph with $n$ vertices, then
    \begin{itemize}
        \item [(i)] the matrix $X= (x_{ij})_{n\times n}= (L+\frac{1}{n}J)^{-1}$ exists,
        \item [(ii)] the eigenvalues of $X$ are $\frac{1}{\gamma_1}, \frac{1}{\gamma_2}, \ldots, \frac{1}{\gamma_{n-1}}, 1$,
        \item [(iii)] the eigenvectors of $X$ coincide with the Laplacian eigenvectors $S_1, S_2, \ldots, S_n$ of $G,$
        \item [(iv)] $x_{ij}=\frac{1}{n}+\sum_{k=1}^{n-1}\frac{1}{\gamma_k}s_{ki}s_{kj},\;\text{where}\; 1\leq i,j\leq n.$
    \end{itemize}
    
\end{theorem}
\begin{proposition}\cite{zhou2016resistance}\label{resreg}
A connected graph $G$ with $n$ vertices is resistance regular if and only if $l_{11}^{\dagger}= \cdots = l_{nn}^{\dagger}.$
\end{proposition}
\begin{theorem}\cite{zhou2016resistance}\label{distinct}
  A connected graph with two distinct $R$-eigenvalues is a complete graph.
\end{theorem}
\begin{lemma}\label{BS} \cite{davis1979circulant}
Let $B= \begin{bmatrix}B_0& B_1\\ B_1& B_0\end{bmatrix}_{2\times 2}$ be a block symmetric matrix. Then the eigenvalues of $B_0+B_1$ and $B_0-B_1$ form the eigenvalues of $B.$
\end{lemma}

\begin{proposition}\label{kf}\cite{maden2013bounds}
For a connected graph $G$ on $n$ ($\geq 2$) vertices,
 $$\rho_1(G)\geq \frac{2\mathcal{K}f(G)}{n}.$$ The equality holds if and only if $G$ is resistance regular.
\end{proposition}
\begin{theorem}\label{twobds}\cite{maden2013bounds}
Let $G$ be a connected graph with $n$ ($\geq 2$) vertices. Then $$\sqrt{\frac{\sum_{i=1}^{n}R_i^2}{n}}\leq \rho_1(G)\leq \max\limits_{1\leq i\leq n}\sum_{i=1}^{n}r_{ij}\frac{R_j}{R_i}.$$
Moreover, the equality holds if and only if $G$ is resistance regular.

\end{theorem}
\begin{theorem}\label{irr}\cite{indulal2008distance}
Let $M$ be a real symmetric irreducible square matrix in  which the sum of each row is the same constant. Then there exists a polynomial $Q(x)$
 such that $Q(M)=J$.
\end{theorem}
\section{Resistance regular graphs}\label{sec3}
This section provides a necessary and sufficient conditions for a graph to be resistance regular.\\
 Consider a connected graph $G$, and let \[ X=\left(L+\frac{1}{n}J\right)^{-1}= \begin{bmatrix}
 x_{11}&x_{12}&\cdots&x_{1n}\\
 x_{12}&x_{22}&\cdots&x_{2n}\\
 \vdots&\vdots&\ddots&\vdots\\
 x_{1n}&x_{2n}&\cdots&x_{nn}
\end{bmatrix}.\]
Then the graph $G$ is resistance regular if and only if $x_{11}=x_{22}=\cdots=x_{nn}=x'$(say) \cite{zhou2016resistance}. All the walk-regular graphs are resistance regular \cite{zhou2016resistance}. Also, connected regular graph with at most four distinct eigenvalues are walk-regular \cite{brouwer2011spectra}. From this we get the following proposition.
\begin{proposition}
    Any regular graph with at most four distinct $A$-eigenvalues is resistance regular.
\end{proposition}

Distance-regular graphs are a part of walk-regular graphs and so they are resistance regular. But there are graphs which are resistance regular but not distance regular. For example see Figure \ref{fig7.11}.

\begin{figure}[H]
    \centering
        \begin{tikzpicture}[scale=0.6,inner sep=1.2pt]
            \draw (0,0) node(4) [circle,draw,fill] (4) {};
            \draw (0,6) node(1) [circle,draw,fill] (1) {};
            \draw (-2,4) node(2) [circle,draw,fill] (2) {};     
            \draw (-2,2) node(3) [circle,draw,fill] (3) {};
            \draw (2,2) node(5) [circle,draw,fill] (5){};
            \draw (2,4) node(6) [circle,draw,fill] (6){};
            \draw (0, -0.2) node[below]{$4$};   
            \draw (0, 6.9) node[below]{$1$}; 
            \draw (-2.4, 4) node[]{$2$}; 
            \draw (-2.4, 2) node[]{$3$};
            \draw (2.4, 2) node[]{$5$};
            \draw (2.4, 4) node[]{$6$};
            \draw [-] (1) to (2) to (3) to (4) to (5) to (6) to (1) to (4);
            \draw[-] (2) to (6);
            \draw[-] (3) to (5);
            \node at (0,-1.5) {$G$};
            \node at (10,2) {\(
        R(G) = \begin{bmatrix}
            0&\frac{8}{15}&\frac{11}{15}&\frac{3}{5}&\frac{11}{15}&\frac{8}{15}\\
            &&&&&\\
            \frac{8}{15}&0&\frac{3}{5}&\frac{11}{15}&\frac{11}{15}&\frac{8}{15}\\
             &&&&&\\
            \frac{11}{15}&\frac{3}{5}&0&\frac{8}{15}&\frac{8}{15}&\frac{11}{15}\\
             &&&&&\\
            \frac{3}{5}&\frac{11}{15}&\frac{8}{15}&0&\frac{8}{15}&\frac{11}{15}\\
             &&&&&\\
            \frac{11}{15}&\frac{11}{15}&\frac{8}{15}&\frac{8}{15}&0&\frac{3}{5}\\
             &&&&&\\
            \frac{8}{15}&\frac{8}{15}&\frac{11}{15}&\frac{11}{15}&\frac{3}{5}&0
        \end{bmatrix}
        \)};
        \end{tikzpicture} 
    \caption{Resistance regular graph on $6$ vertices with matrix representation}
    \label{fig7.11}
\end{figure}
All the regular graphs need not be resistance regular, for example one can look into the graph in Figure \ref{fig7.1}.

    \begin{figure}[H]
    \centering
    \begin{tabular}{cc}
        \begin{tikzpicture}[scale=0.6,inner sep=1.2pt]
            \draw (0,0) node(3) [circle,draw,fill] {};
            \draw (0,3) node(7) [circle,draw,fill] {};
            \draw (0,6) node(1) [circle,draw,fill] {};
            \draw (-1,2) node(8) [circle,draw,fill] {};
            \draw (1,2) node(9) [circle,draw,fill] {};
            \draw (-1,4) node(5) [circle,draw,fill] {};
            \draw (1,4) node(6) [circle,draw,fill] {};
            \draw (-3,3) node(2) [circle,draw,fill] {};
            \draw (3,3) node(4) [circle,draw,fill] {};
            \node at (0,-1.5) {$G$};
            \node at (13,3) {\(
        R(G)= \begin{bmatrix}
            0&\frac{41}{90}&\frac{3}{5}&\frac{41}{90}&\frac{19}{45}&\frac{19}{45}&\frac{17}{30}&\frac{28}{45}&\frac{28}{45}\\
             &&&&&&&&\\
            \frac{41}{90}&0&\frac{41}{90}&\frac{5}{9}&\frac{41}{90}&\frac{17}{30}&\frac{5}{9}&\frac{41}{90}&\frac{17}{30}\\
            &&&&&&&&\\
            \frac{3}{5}&\frac{41}{90}&0&\frac{41}{90}&\frac{28}{45}&\frac{28}{45}&\frac{17}{30}&\frac{19}{45}&\frac{19}{45}\\
            &&&&&&&&\\
            \frac{41}{90}&\frac{5}{9}&\frac{41}{90}&0&\frac{17}{30}&\frac{41}{90}&\frac{5}{9}&\frac{17}{30}&\frac{41}{90}\\
            &&&&&&&&\\
            \frac{19}{45}&\frac{41}{90}&\frac{28}{45}&\frac{17}{30}&0&\frac{19}{45}&\frac{41}{90}&\frac{3}{5}&\frac{28}{45}\\
            &&&&&&&&\\
            \frac{19}{45}&\frac{17}{30}&\frac{28}{45}&\frac{41}{90}&\frac{19}{45}&0&\frac{41}{90}&\frac{28}{45}&\frac{3}{5}\\
            &&&&&&&&\\
            \frac{17}{30}&\frac{5}{9}&\frac{17}{30}&\frac{5}{9}&\frac{41}{90}&\frac{41}{90}&0&\frac{41}{90}&\frac{41}{90}\\
            &&&&&&&&\\
            \frac{28}{45}&\frac{41}{90}&\frac{19}{45}&\frac{17}{30}&\frac{3}{5}&\frac{28}{45}&\frac{41}{90}&0&\frac{19}{45}\\
            &&&&&&&&\\
            \frac{28}{45}&\frac{17}{30}&\frac{19}{45}&\frac{41}{90}&\frac{28}{45}&\frac{3}{5}&\frac{41}{90}&\frac{19}{45}&0
        \end{bmatrix}.
        \)};
            \draw[-] (6) to (1) to (2) to (3) to (4) to (1) to (5) to (6) to (7) to (8) to (9) to (3) to (8) to (2) to (5) to (6) to (4) to (9) to (7) to (5);
        \end{tikzpicture} &
        
    \end{tabular}

    \caption{Regular and non-resistance regular graph (row sums of 
$R(G)$ are not all equal) with matrix representation}
    \label{fig7.1}
\end{figure}
\begin{proposition}
    The cocktail party graph $CP(n)$ of order $n=2p$ is resistance regular, where $p\in \mathbb{N}$.
\end{proposition}
\begin{proof}
    By a proper labelling of vertices in $CP(n)$ its Laplacian matrix can be written as,
    $$L(CP(n))= 2(p-1)I_n+(I_p-J_p)\otimes J_2.$$
    Let $X= -\frac{1}{4p(p-1)}(I_p\otimes J_2)+\frac{1}{n-2}I_n-\frac{1}{n^2}J_n.$ By computation, we have
    \begin{equation*}
        \begin{aligned}
         \left(2(p-1)I_n+(I_p-J_p)\otimes J_2\right)X\left(2(p-1)I_n+(I_p-J_p)\otimes J_2\right)&= 2(p-1)I_n+(I_p-J_p)\otimes J_2,\\
         \end{aligned}
    \end{equation*}
    \begin{equation*}
        \begin{aligned}
         X\left(2(p-1)I_n+(I_p-J_p)\otimes J_2\right)X&= X,
         \end{aligned}
    \end{equation*}
    \begin{equation*}
        \begin{aligned}
         \left(2(p-1)I_n+(I_p-J_p)\otimes J_2\right)X= I_n-\frac{1}{n}J_n= X\left(2(p-1)I_n+(I_p-J_p)\otimes J_2\right).
        \end{aligned}
    \end{equation*}
    Therefore, $$L^{\dagger}(CP(n))= -\frac{1}{4p(p-1)}(I_p\otimes J_2)+\frac{1}{n-2}I_n-\frac{1}{n^2}J_n.$$
    We can see that $$l^{\dagger}_{11}= \cdots= l^{\dagger}_{nn}= \frac{n^2(2p-3)+4(n-1)-2p(n-2)}{2n^2(n-2)(p-1)}.$$\\ Then by Proposition \ref{resreg}, $CP(n)$ is resistance regular.
\end{proof}
Next theorem provides a necessary and sufficient condition for a graph to be a resistance regular graph.
\begin{theorem}
    Let $G$ be a graph with $n$ vertices, then $G$ is resistance regular if and only if\\ $\frac{1}{\mathrm{deg}_G(v_i)}\left(\sum\limits_{v_j\in N(v_i)}\left(\sum_{k=1}^{n-1}\frac{1}{\gamma_k}s_{k_i}s_{kj}\right)+1-\frac{1}{n}\right)$ is same for every $i\in \{1,2,\ldots,n\},$ where $s_{ij}$'s are the components of the eigenvector $S_i$ of the eigenvalue $\gamma_i$ of $L(G).$ 
\end{theorem}
\begin{proof}
    Consider $L^{\dagger}= (l^{\dagger}_{ij})= (L+\frac{1}{n}J)^{-1}-\frac{1}{n}J.$ From Proposition \ref{resreg}, we have the graph $G$ is resistance regular if and only if $l^{\dagger}_{11}=l^{\dagger}_{22}=\cdots=l^{\dagger}_{nn}.$\\
    From Lemma \ref{lldagger}, $LL^{\dagger}= I-\frac{1}{n}J,$ and so \begin{equation*}
        \begin{aligned}
            \mathrm{deg}_G(v_i)l_{ii}^{\dagger}-\sum\limits_{v_j\in N(v_i)}l_{ij}^{\dagger}&= 1-\frac{1}{n}\\
            l_{ii}^{\dagger}&= \frac{\sum\limits_{v_j\in N(v_i)}l_{ij}^{\dagger}+1-\frac{1}{n}}{\mathrm{deg}_G(v_i)}
        \end{aligned}
    \end{equation*}
    That is, $G$ is resistance regular if and only if \begin{equation}\label{eqn7.111}
        \frac{\sum\limits_{v_j\in N(v_1)}l^{\dagger}_{1j}+1-1/n}{\mathrm{deg}_G(v_1)}=\cdots= \frac{\sum\limits_{v_j\in N(v_n)}l^{\dagger}_{nj}+1-1/n}{\mathrm{deg}_G(v_n)}.
    \end{equation}\\
    Now by Theorem \ref{evectorformX}, the equation (\ref{eqn7.111}) holds if and only if
     $$\frac{1}{\mathrm{deg}_G(v_1)}\left(\sum_{{\substack{v_j\in N(v_1)}}}\left(\sum_{k=1}^{n-1}\frac{1}{\gamma_k}s_{k1}s_{kj}\right)+1-\frac{1}{n}\right)= \cdots=\frac{1}{\mathrm{deg}_G(v_n)}\left(\sum_{v_j\in N(v_n)}\left(\sum_{k=1}^{n-1}\frac{1}{\gamma_k}s_{kn}s_{kj}\right)+1-\frac{1}{n}\right).$$
     
\end{proof}

\section{Resistance Energy}
In this section we determine the $R$-energy of some standard resistance regular graphs.
Further, establishes sharp bounds for the resistance spectral radius, and presents various
bounds for the resistance energy of graphs.
Note that $R(G)$ has a single positive eigenvalue and $n-1$ negative eigenvalues. Then we have,
$$E_R(G)= 2\rho_1(G).$$\\
The following propositions are direct consequence of Theorem \ref{twobds} and Proposition \ref{kf}.
\begin{proposition}
    For a graph $G$ with $n$ vertices, 
    $$E_R(G)\geq 2\sqrt{\frac{1}{n}\sum_{i=1}^{n}R_i^2}.$$
  
Equality holds if and only if $R(G)$ is resistance regular.
\end{proposition}

\begin{proposition}
For a graph $G$,    $E_R(G)\geq \frac{4\mathcal{K}f(G)}{n}$ equality holds if and only if $G$ is resistance regular.
\end{proposition}

\begin{corollary}\label{cor7.3.1}
     If $G$ is $k$-resistance regular, then $E_R(G)= 2k.$
Also $\mathcal{K}f(G)= \frac{nk}{2}.$

\end{corollary}
\begin{remark}
    All resistance regular graphs have the same Kirchhoff index.
\end{remark}
Using Corollary \ref{cor7.3.1}, the resistance energy for certain well-known graphs are obtained as follows.
\begin{corollary}
The resistance energy of some standard resistance regular graphs are given as follows:
    \begin{itemize}
        \item [(i)] If $G= K_n$, then $E_R(G)= \frac{4(n-1)}{n}.$
        \item[(ii)] If $G= C_n$, then $E_R(G)= \frac{n^2-1}{3}.$
        \item[(iii)] If $G= K_{n,n}$, then $E_R(G)= \frac{8n-6}{n}.$
    \end{itemize}
\end{corollary}
\begin{proof}
    \begin{itemize}
        \item [(i)] Let $G= K_n$, then $R(G)= \frac{2}{n}(J-I.)$ That is, $K_n$ is $\frac{2(n-1)}{n}$-resistance regular.\\
        By
        Corollary \ref{cor7.3.1}, $E_R(G)= \frac{4(n-1)}{n}.$
        \item[(ii)] Let $G= C_n$. Then from Lemma \ref{rescycle}, for any $v_i, v_j \in V(G),\; r_{ij} = \frac{d_{ij}(n-d_{ij})}{n}.$\\
        If $n$ is even, then $G$ is $\frac{n^2}{4}$-transmission regular.\\
        Now, the $i^{th}$ resistance degree is given by
        \begin{equation*}
            \begin{aligned}
                R_i&= \sum_{j=1}^{n}r_{ij}\\
                &= \sum_{j=1}^{n}d_{ij}-\frac{1}{n}\sum_{j=1}^{n}d_{ij}^2\\
             &= \frac{n^2}{4}-\frac{1}{n}\left(2\left(1^2+2^2+\cdots+\left(\frac{n}{2}-1\right)^2\right)+\left(\frac{n}{2}\right)^2\right)\\
             &= \frac{n^2-1}{6}.
            \end{aligned}
        \end{equation*}
        Now, by Corollary \ref{cor7.3.1}, $E_R(G)= \frac{n^2-1}{3}.$\\
        If $n$ is odd, then $G$ is $\frac{n^2-1}{4}$-transmission regular.\\
        Therefore, the $i^{th}$ resistance degree is given by
        \begin{equation*}
            \begin{aligned}
                R_i&= \sum_{j=1}^{n}d_{ij}-\frac{1}{n}\sum_{j=1}^{n}d_{ij}^2\\
             &= \frac{n^2-1}{4}-\frac{2}{n}\left(1^2+2^2+\cdots+\left(\frac{n-1}{2}\right)^2\right)\\
             &= \frac{n^2-1}{6}.
            \end{aligned}
        \end{equation*}
        By Corollary \ref{cor7.3.1}, $E_R(G)= \frac{n^2-1}{3}.$
        \item[(iii)] Let $G= K_{n,n}.$ Then \[L(G)= \begin{bmatrix}
            nI_n & -J_n\\
            -J_n & nI_n
        \end{bmatrix}.\]
        By Lemma \ref{L},
    \[L^{(1)}(G)= \begin{bmatrix}
        \frac{1}{n}I_n& 0\\
        0& (nI_n-J_n)^{\dagger}
    \end{bmatrix}.\]
Consider $X=\frac{1}{n^2}(nI_n-J_n).$ Then\\
\begin{equation*}
    \begin{aligned}
        X(nI_n-J_n)X&= X,\\
        (nI_n-J_n)X(nI_n-J_n)&= nI_n-J_n,\\
        X(nI_n-J_n)&= \frac{1}{n}(nI_n-J_n)= X(nI_n-J_n).
    \end{aligned}
\end{equation*}
Therefore, $(nI_n-J_n)^{\dagger}= \frac{1}{n^2}(nI_n-J_n).$
\[L^{(1)}(G)= \begin{bmatrix}
    \frac{1}{n}I_n& 0\\
    0& \frac{1}{n^2}(nI_n-J_n)
\end{bmatrix}.\]
Now by the definition of resistance distance \[R(G)= \begin{bmatrix}
     \frac{2}{n}(J_n-I_n)& \frac{2n-1}{n^2}J_n\\
\frac{2n-1}{n^2}J_n&\frac{2}{n}(J_n-I_n)
\end{bmatrix}.\]
Then the $i^{th}$ resistance distance degree is given by $R_i= \frac{4n-3}{n}.$\\
Now by Corollary \ref{cor7.3.1}, $E_R(G)= \frac{8n-6}{n}.$
     \end{itemize}
\end{proof}
\begin{proposition}
Let $R_i$ and $T_i$ denote the resistance degrees and the second resistance degrees of $G$, respectively, then
    $$T_1+\cdots +T_n= R_1^2+\cdots +R_n^2.$$
\end{proposition}
\begin{proof}
    By definition $T_i= \sum_{j=1}^{n}r_{ij}R_j$.\\
    Now \begin{equation*}
        \begin{aligned}
            T_1+\cdots +T_n&=\sum_{j=1}^{n}r_{1j}R_j+\cdots+\sum_{j=1}^{n}r_{nj}R_j\\
            &=R_1(r_{11}+r_{21}+\cdots+r_{n1})+\cdots+R_n(r_{1n}+\cdots+r_{nn})\\
            &=R_1^2+\cdots+R_n^2.
        \end{aligned}
    \end{equation*}
\end{proof}
\begin{theorem}\label{ti}
  Let $R_i$ denote the resistance degrees of a graph $G$ for $i= 1,\ldots, n$, then  $$\rho_1(G)\geq \sqrt{\frac{T_1^2+\cdots +T_n^2}{R_1^2+\cdots+ R_n^2}}.$$
 Equality holds if and only if $G$ is pseudo resistance regular.
\end{theorem}
\begin{proof}
Let $Z=\left(z_1,\ldots,z_n\right)^T$ be the unit positive Perron eigenvector of $R(G)$ corresponding to $\rho_1(G).$\\

    Take $$U= \frac{1}{\sqrt{\sum_{i=1}^{n}R_i^2}}\left(R_1, \ldots, R_n\right)^T.$$\\
    Then $U$ is a unit positive vector.\\
  We have, $$\rho_1(G)= \sqrt{\rho_1(G)^2}=\sqrt{Z^TR^2Z}\geq \sqrt{U^TR^2U}.$$\\
    \begin{equation*}
    \begin{aligned}
        RU&= [r_{ij}]_{n\times n}\frac{1}{\sqrt{\sum_{i=1}^{n}R_i^2}}\left(R_1, R_2, \ldots, R_n\right)^T\\&= \frac{1}{\sqrt{\sum_{i=1}^{n}R_i^2}}\left(T_1, T_2, \ldots, T_n\right)^T,\\
       U^TR^2U= (RU)^T(RU)&= \frac{T_1^2+\cdots +T_n^2}{R_1^2+\ldots+ R_n^2}, 
        \end{aligned}
    \end{equation*}
    then $$\rho_1(G)\geq \sqrt{\frac{T_1^2+\cdots +T_n^2}{R_1^2+\cdots+ R_n^2}}.$$
    Now assume that $G$ is pseudo resistance regular. So $\frac{T_i}{R_i}
 $ is a constant (say $l$) for all $i$. Then $RU= lU$, showing that $U$ is an eigenvector corresponding to $l$ and hence $\rho_1(G)= l$. Thus, the equality holds.
 Conversely if equality holds then, we get $U$ is the eigenvector corresponding to $\rho_1(G)$ and that $RU= \rho_1(G)U$. This implies that $\frac{T_i}{R_i}
 = \rho_1(G)$ or in other words $G$ is
 pseudo resistance regular. 
 \end{proof}
We present some upper and lower bounds for the resistance energy of a graph $G$. The subsequent corollary is an immediate consequence of Theorem \ref{ti}.

\begin{corollary}
    For a graph $G$,
    $$E_R(G)\geq 2\sqrt{\frac{T_1^2+\cdots +T_n^2}{R_1^2+\cdots+ R_n^2}}.$$
Equality holds if and only if $G$ is pseudo resistance regular.
\end{corollary}
\noindent Consider the $R$-eigenvalues of $G$, $\rho_1(G), \ldots, \rho_n(G)$, then $\sum_{i=1}^{n}\rho_i(G)= 0.$ Let $S(G)= \sum_{i=1}^{n}\rho_i(G)^2= \sum_{i=1}^{n}\sum_{j=1}^{n}r_{ij}^2.$
\begin{theorem}
    For a graph $G,$
    \begin{equation}\label{4.2ineq}
        E_R(G)\leq \alpha(G)+\sqrt{(n-1)\left(S(G)-\alpha(G)\right)},
    \end{equation}
    where $\alpha(G)= \sqrt{\frac{\sum_{i=1}^{n}T_i^2}{\sum_{i=1}^{n}R_i^2}}.$
    Equality holds if and only if $G$ is a complete graph.
\end{theorem}
\begin{proof}
    Consider \begin{equation}\label{cauchyschwartz}
        \begin{aligned}
            (E_R(G)-\rho_1(G))^2= \left(\sum_{i=2}^{n}|\rho_i(G)|\right)^2&\leq (n-1)\sum_{i=2}^{n}(\rho_i(G))^2\\
            &= (n-1)(S(G)-(\rho_1(G))^2)
        \end{aligned}
    \end{equation}
    Thus, $E_R(G)\leq \rho_1(G)+\sqrt{(n-1)(S(G)-(\rho_1(G))^2)}.$\\

     Now consider the function $g(x)= x+ \sqrt{(n-1)(S(G)-x^2)}$ for $\frac{2Kf(G)}{n}\leq x \leq \sqrt{S(G)}.$\\
     The function $g(x)$ is monotonically decreasing for $x\geq \sqrt{\frac{S(G)}{n}}.$\\
     For any $i\in \{1,\ldots, n\}$, \\
     \begin{equation*}
         \begin{aligned}
             R_i^2&= \left(\sum_{j=1}^{n}r_{ij}\right)^2\leq n\sum_{j=1}^{n}r_{ij}^2\\
             \sum_{i=1}^{n}R_i^2&\leq n\sum_{i=1}^{n}\sum_{j=1}^{n}r_{ij}^2= nS(G).
         \end{aligned}
     \end{equation*}

\begin{equation*}
    \begin{aligned}
        \sum_{i=1}^{n}T_i^2= \sum_{i=1}^{n}\left(\sum_{j=1}^{n}r_{ij}R_j\right)^2\geq S(G)^2.
    \end{aligned}
\end{equation*}
Then by Theorem \ref{ti},\\
$$\rho_1(G)\geq \sqrt{\frac{\sum_{i=1}^{n}T_i^2}{\sum_{i=1}^{n}R_i^2}}\geq \sqrt{\frac{S(G)}{n}}.$$
    Therefore, $E_R(G)\leq g(\rho_1(G))\leq g\left(\sqrt{\frac{\sum_{i=1}^{n}T_i^2}{\sum_{i=1}^{n}R_i^2}}\right).$\\
    Suppose equality holds in (\ref{4.2ineq}), then $\rho_1(G)= \sqrt{\frac{\sum_{i=1}^{n}T_i^2}{\sum_{i=1}^{n}R_i^2}}$, by Theorem \ref{ti} $G$ is pseudo resistance regular. Applying equality condition in Cauchy-Schwarz inequality, we get
    \begin{equation*}
        \begin{aligned}
            \left(\sum_{i=2}^{n}|\rho_i(G)|\right)^2&= (n-1)\sum_{i=2}^{n}(\rho_i(G))^2.
        \end{aligned}
    \end{equation*}
    Hence, $|\rho_2(G)|=\cdots= |\rho_n(G)|.$
Further, from (\ref{cauchyschwartz}), we have \begin{equation*}
    \begin{aligned}
        \left((n-1)|\rho_i(G)|\right)^2&= (n-1)\sum_{i=2}^{n}(\rho_i(G))^2= (n-1)\left(S(G)-(\rho_1(G))^2\right),\\
        |\rho_i(G)|&= \sqrt{\frac{S(G)-(\rho_1(G))^2}{n-1}}, \; i\in\{2,3,\ldots, n\}.
    \end{aligned}
\end{equation*}
That is, $G$ has at most two distinct 
$R$-eigenvalues. Therefore, by Theorem \ref{distinct}, $G$ is a complete graph. Hence the theorem.
\end{proof}
\begin{theorem}\label{upperbd}
Let $\overline{R_i}$ denotes the average resistance degree of a graph $G$, then
    $$\rho_1(G)\leq \max\limits_{i\neq j}\sqrt{{\overline{R_i}}\,{\overline{R_j
}}}.$$
Equality holds if and only if $G$ is pseudo resistance regular.
\end{theorem}
\begin{proof}
    Let $T(G)= diag(R_1,\ldots, R_n)$ and $M= T(G)^{-1}RT(G)$, then the $(i,j)^{th}$ element of $M$ is $\frac{R_j}{R_i}r_{ij}.$ Now let $y= \left(y_1,\ldots, y_n\right)$ be the eigenvector of $T(G)^{-1}RT(G)$ corresponding to $\rho_1(G),$ with \begin{equation*}
        \begin{aligned}
            y_p&= \mathrm{max}\{y_i:1\leq i\leq n\},\\
            y_q&= \mathrm{max}\{y_i:y_i\neq y_p, 1\leq i\leq n\}.
        \end{aligned}
    \end{equation*} 
   Also,
    \begin{equation*}
        \begin{aligned}
            \left(T(G)^{-1}RT(G)\right)y&= \rho_1(G) y.
            \end{aligned}
            \end{equation*}

\noindent Then
            \begin{equation*}
        \begin{aligned}
            \rho_1(G) y_p= \sum_{i=1}^{n}\frac{R_i}{R_p}r_{pi}y_i &\leq \sum_{i=1}^{n}\frac{R_i}{R_p}r_{pi}y_q\\
            &= \frac{T_p}{R_p}y_q= \overline{R}_p y_q.
            \end{aligned}
            \end{equation*}
            That is, \begin{equation}\label{eqn4}
                \rho_1(G)y_p\leq \overline{R}_py_q.
            \end{equation}
           Similarly, \begin{equation*}
            \begin{aligned}
            \rho_1(G) y_q= \sum_{i=1}^{n}\frac{R_i}{R_q}r_{qi}y_i&\leq \sum_{i=1}^{n}\frac{R_i}{R_q}r_{qi}y_p\\
            &= \frac{T_q}{R_q}y_p= \overline{R}_qy_p.
        \end{aligned}
    \end{equation*}
    Thus, \begin{equation}\label{eqn5}
        \rho_1(G)y_q\leq \overline{R}_qy_p.
    \end{equation}
    From equations (\ref{eqn4}) and (\ref{eqn5}), we get \begin{equation*}
        \begin{aligned}
            \rho_1(G)^2 y_py_q&\leq \overline{R}_p\overline{R}_q y_p y_q,\\
            \rho_1(G)&\leq \max\limits_{i\neq j}\sqrt{\overline{R}_i\overline{R}_j}.
        \end{aligned}
    \end{equation*}

Suppose equality holds, then $\overline{R}_1= \cdots = \overline{R}_n.$ That is $G$ is pseudo resistance regular. Conversely, if $G$ is pseudo resistance regular, then $\overline{R}_1= \cdots = \overline{R}_n=c$. Then
the $i^{th}$ row sum of $M$ is $\sum_{j=1}^{n}\frac{R_j}{R_i}r_{ij}= \overline{R}_i=c,\; i\in \{1,\ldots , n\}.$
Therefore, $\rho_1(M)= \rho_1(R)= c$, since the matrices $M$ and $R$ are similar. Hence the theorem.

\end{proof}
\begin{theorem}\label{lowerbd}
Let $\overline{R_i}$ denotes the average resistance degree of a graph $G$, then
    $$\rho_1(G)\geq \min\limits_{i\neq j}\sqrt{\overline{R_i}\,\overline{R_j
}}.$$
Equality holds if and only if $G$ is pseudo resistance regular.
\end{theorem}
\begin{proof}
    Let $T(G)= diag(R_1,\ldots, R_n)$ and $M= T(G)^{-1}RT(G)$, then the $(i,j)^{th}$ element of $M$ is $\frac{R_j}{R_i}r_{ij}.$ Now let $y= \left(y_1,\ldots, y_n\right)$ be the eigenvector of $T(G)^{-1}RT(G)$ corresponding to $\rho_1(G),$ with \begin{equation*}
        \begin{aligned}
            y_p&= \min\{y_i:1\leq i\leq n\},\\
            y_q&= \min\{y_i:y_i\neq y_p, 1\leq i\leq n\}.
        \end{aligned}
    \end{equation*} 
   Also,
    \begin{equation*}
        \begin{aligned}
            \left(T(G)^{-1}RT(G)\right)y&= \rho_1(G) y.
            \end{aligned}
            \end{equation*}
            \noindent Then
             \begin{equation*}
        \begin{aligned}
            \rho_1(G) y_p= \sum_{i=1}^{n}\frac{R_i}{R_p}r_{pi}y_i &\geq \sum_{i=1}^{n}\frac{R_i}{R_p}r_{pi}y_q\\
            &= \frac{T_p}{R_p}y_q= \overline{R}_p y_q.
            \end{aligned}
            \end{equation*}
            That is, \begin{equation}\label{eqn6}
                \rho_1(G)y_p\geq \overline{R}_py_q.
            \end{equation}
           Similarly, \begin{equation*}
            \begin{aligned}
            \rho_1(G) y_q= \sum_{i=1}^{n}\frac{R_i}{R_q}r_{qi}y_i&
            \geq \sum_{i=1}^{n}\frac{R_i}{R_q}r_{qi}y_p\\
            &= \frac{T_q}{R_q}y_p= \overline{R}_qy_p.
        \end{aligned}
    \end{equation*}
    Thus \begin{equation}\label{eqn7}
        \rho_1(G)y_q\geq \overline{R}_qy_p.
    \end{equation}
    From equations (\ref{eqn6}) and (\ref{eqn7}), we get \begin{equation*}
        \begin{aligned}
            \rho_1(G)^2 y_py_q&\geq \overline{R}_p\overline{R}_q y_p y_q,\\
            \rho_1(G)&\geq \min\limits_{i\neq j}\sqrt{\overline{R}_i\overline{R}_j}.
        \end{aligned}
    \end{equation*}

The equality holds if and only if $G$ is pseudo resistance regular, proof follows by similar arguments in Theorem \ref{upperbd}.

\end{proof}
From Theorems \ref{upperbd} and \ref{lowerbd}, we get the lower and upper bound for the resistance energy of $G$ in terms of average resistance degrees.
\begin{corollary}
    For a graph $G$,
    $$2 \min\limits_{i\neq j}\sqrt{\overline{R_i}\overline{R_j
}}\leq E_R(G)\leq 2 \max\limits_{i\neq j}\sqrt{\overline{R_i}\overline{R_j
}}.$$
\end{corollary}

\section{Resistance spectrum}
This section discusses $R$-spectrum of certain graphs. The following theorem directly follows from Theorem \ref{irr}.
\begin{proposition}\label{RJ}
    Let $R(G)$ be the resistance distance matrix of a $k$-resistance regular graph $G$ with $n$ vertices, then there exists a polynomial $Q(x)$ such that $Q(R(G))= J$. In this case $$Q(x)= n\frac{(x-\rho_2(G))\ldots(x-\rho_t(G))}{(k-\rho_2(G))\ldots (k-\rho_t(G))},$$ where $\rho_1(G), \ldots, \rho_t(G)$ are the distinct $R$-eigenvalues of $G$.
\end{proposition}

Let $\{u_1,\ldots, u_n, u_{1}',\ldots, u_{n}'\}$ be the vertex set of $D_2G$. From Theorem \ref{resdouble}, the resistance distance between two vertices $v_i$ and $v_j$ in $V(D_2G)$ is as follows,
     $$r_{ij}(D_2G)=\begin{cases}\frac{1}{r}, &\text{if $v_i=u_i$ and $v_j=u_{i}'$,}\\
     0, &\text{if $v_i= v_j$,}\\\frac{1}{4}\left(r_{ij}(G)+\frac{2}{r}\right), &\text{otherwise.}\end{cases}$$
     In the following theorem, we obtain the resistance eigenvalues of the double graph of a resistance regular graph.
\begin{theorem}
    Let $G$ be an $r$-regular and $k$-resistance regular graph of order $n$ with $R$-eigenvalues $\rho_1(G), \ldots, \rho_n(G)$. Then the $R$-eigenvalues of $D_2G$ is as follows,
    $$Spec_{R}(D_2G)= \begin{pmatrix}
        \frac{\rho_1(G)}{2}+\frac{n}{r}&\frac{\rho_2(G)}{2}&\ldots&\frac{\rho_n(G)}{2}&-\frac{1}{r}\\
        1&1&\ldots &1&n
    \end{pmatrix}.$$
\end{theorem}
\begin{proof}
By a proper labelling of vertices in $D_2G$ the resistance distance matrix of $D_2G$ can be written as
     \[R(D_2G)= \begin{bmatrix}
         \frac{1}{4}R+\frac{1}{2r}(J-I)&\frac{1}{4}R+\frac{1}{2r}(J+I)\\&\\
         \frac{1}{4}R+\frac{1}{2r}(J+I)&\frac{1}{4}R+\frac{1}{2r}(J-I)
     \end{bmatrix}.\]
From Lemma \ref{BS} and Proposition \ref{RJ}, we have 
$$Spec_{R}(D_2G)= \begin{pmatrix}
        \frac{\rho_1(G)}{2}+\frac{n}{r}&\frac{\rho_2(G)}{2}&\cdots&\frac{\rho_n(G)}{2}&-\frac{1}{r}\\
        1&1&\cdots &1&n
    \end{pmatrix}.$$
     
\end{proof}
\begin{corollary}
    For a $k$-resistance regular and $r$-regular graph of order $n$,
    $E_R(D_2G)= \rho_1(G)+\frac{2n}{r}.$
\end{corollary}
\begin{corollary}
    If $G$ is $k$-resistance regular and $r$-regular, then $D_2G$ is $\left(\frac{k}{2}+\frac{n}{r}\right)$-resistance regular.
\end{corollary}

Consider $G[K_2].$ From Proposition \ref{reslexico}, the resistance distance between two vertices $w_i= (v_i, u_i)$ and $w_j= (v_j, u_j)$ is as follows,
$$r_{ij}(G[K_2])=\begin{cases}\frac{1}{r+1}, &\text{if $v_i=v_j$ and $u_i\neq u_j$,}\\
     \frac{r_{ij}(G)}{4}+\frac{1}{2(r+1)}, &\text{otherwise.}\end{cases}$$
     The next theorem provides the $R$-eigenvalues of $G[K_2]$ when $G$ is resistance regular.
\begin{theorem}
    Let $G$ be an $r$-regular and $k$-resistance regular graph of order $n$ with $R$-eigenvalues $\rho_1(G), \ldots, \rho_n(G)$. Then 
    $$Spec_{R}(G[K_2])= \begin{pmatrix}
        \frac{\rho_1(G)}{2}+\frac{n}{r+1}&\frac{\rho_2(G)}{2}&\ldots&\frac{\rho_n(G)}{2}&-\frac{1}{r+1}\\
        1&1&\ldots &1&n
    \end{pmatrix}.$$
\end{theorem}
\begin{proof}
 By a proper labelling of vertices in $G[K_2]$ the resistance distance matrix of $G[K_2]$ can be written as
     \[R(G[K_2])= \begin{bmatrix}
         \frac{1}{4}R(G)+\frac{1}{2(r+1)}(J-I)&\frac{1}{4}R(G)+\frac{1}{2(r+1)}(J+I)\\&\\
         \frac{1}{4}R(G)+\frac{1}{2(r+1)}(J+I)&\frac{1}{4}R(G)+\frac{1}{2(r+1)}(J-I)
     \end{bmatrix}.\]
From Lemma \ref{BS} and Proposition \ref{RJ}, we have 
$$Spec_{R}(G[K_2])= \begin{pmatrix}
        \frac{\rho_1(G)}{2}+\frac{n}{r+1}&\frac{\rho_2(G)}{2}&\ldots&\frac{\rho_n(G)}{2}&-\frac{1}{r+1}\\
        1&1&\ldots &1&n
    \end{pmatrix}.$$
\end{proof}
\begin{corollary}
    For a $k$-resistance regular and $r$-regular graph of order $n$,
    $E_R(G[K_2])= \rho_1(G)+\frac{2n}{r+1}.$
\end{corollary}
\begin{corollary}
    If $G$ is $k$-resistance regular and $r$-regular then $G[K_2]$ is $\left(\frac{k}{2}+\frac{n}{r+1}\right)$-resistance regular.
\end{corollary}
Let $G_1 = (V(G_1), E(G_1))$ and $G_2 = (V(G_2), E(G_2))$ with $V(G_1)=\{v_i:1\leq i\leq n\}$, $V(G_2)=\{u_i:1\leq i\leq n\}.$ Then the \textit{cartesian product} $G_1\times G_2$ is a graph with $V(G_1\times G_2)= V(G_1)\times V(G_2)$ in which two vertices ($v_1, u_1$) and ($v_2, u_2$) are adjacent if $v_1= v_2$ and ($u_1, u_2$) $\in  E(G_2)$ or $u_1 = u_2$ and
($v_1, v_2$) $\in E(G_1)$.
\begin{theorem}
    Let $G=K_n$, then
   $$Spec_{R}(G\times K_2)= \begin{pmatrix}
        \frac{5n^2+2n-4}{n(n+2)}&-\frac{2}{n}&-1&-\frac{2}{n+2}\\1&n-1&1&n-1
    \end{pmatrix}.$$
\end{theorem}
\begin{proof}
We have, \[L(K_n\times K_2)= \begin{bmatrix}
    (n+1)I-J & -I\\-I & (n+1)I-J
\end{bmatrix}.\]
Let $L_1= (n+1)I-J= L_3$ and $L_2=-I$, then \begin{equation*}
    \begin{aligned}
        H&= L_1-L_2^TL_3^{-1}L_2\\
        &= \frac{n+2}{n+1}(nI-J)\\
        H^{\dagger}&= \frac{n+1}{n^2(n+2)}(nI-J)\\
        -H^{\dagger}L_2L_3^{-1}&= \frac{1}{n^2(n+2)}\left(nI-J\right)\\
        L_3^{-1}+L_3^{-1}L_2^TH^{\dagger}L_2L_3^{-1}&=\frac{1}{n^2(n+2)}\left(n(n+1)I+\frac{n^3+2n^2-1}{n+1}J\right).
    \end{aligned}
\end{equation*}
From Lemma \ref{pinv},
\[L^{(1)}=\frac{1}{n^2(n+2)}\begin{bmatrix}
    (n+1)(nI-J)& nI-J\\
    nI-J &  n(n+1)I+\frac{n^3+2n^2-1}{n+1}J
\end{bmatrix}.\]

    By the formula of $r_{ij}$, the resistance distance matrix of $K_n\times K_2$ can be written as
    \[R(K_n\times K_2)= \begin{bmatrix}
        \frac{n+1}{n+2}R(K_n)& \frac{1}{n+2}(3J+R(K_n))\\&\\\frac{1}{n+2}(3J+R(K_n))&\frac{n+1}{n+2}R(K_n)
    \end{bmatrix}.\]
    From Lemma \ref{BS} and Proposition \ref{RJ}, we have 
$$Spec_{R}(K_n\times K_2)= \begin{pmatrix}
        \frac{5n^2+2n-4}{n(n+2)}&-\frac{2}{n}&-1&-\frac{2}{n+2}\\1&n-1&1&n-1
    \end{pmatrix}.$$
\end{proof}
\begin{corollary}
    $E_R(K_n\times K_2)= \frac{10n^2+4n-8}{n(n+2)}.$
\end{corollary}
\begin{corollary}
    For any $n,$ $K_n\times K_2$ is $\left(\frac{5n^2+2n-4}{n(n+2)}\right)$-resistance regular.
\end{corollary}


Next theorem gives an upper bound for the Kirchoff index of a graph in terms of its number of vertices and $R$-eigenvalues.
\begin{theorem}
For a graph $G$ with $n$ vertices,    $$\mathcal{K}f(G)\leq \frac{\sqrt{n(n-1)\sum_{i=1}^{n}\rho_i(G)^2}}{2}.$$
\end{theorem}
\begin{proof}
    We have 
    \begin{equation*}
        \begin{aligned}
            \sum_{i=1}^{n}\rho_i(G)^2&= \operatorname{trace}(R(G)^2)=\sum_{\substack{j=1\\i\neq j}}^{n}r_{ij}^2= 2\sum_{i<j}r_{ij}^2.
        \end{aligned}
    \end{equation*}
    Then  \begin{equation*}
    \begin{aligned}
        (\mathcal{K}f(G))^2= (\sum_{i<j}r_{ij})^2 &\leq \frac{n(n-1)}{2}\sum_{i<j}r_{ij}^2\\
        &= \frac{n(n-1)}{4}\sum_{i=1}^{n}\rho_i(G)^2.
        \end{aligned}
    \end{equation*}
    Therefore, $$\mathcal{K}f(G)\leq \frac{\sqrt{n(n-1)\sum_{i=1}^{n}\rho_i(G)^2}}{2}.$$
\end{proof}
\section{Conclusion}
  This manuscript provides a necessary and sufficient condition for a graph to be a resistance regular graph. Additionally, various bounds for the $R$-energy and $\rho_1(G)$ of $G$ are determined. Furthermore, the $R$-spectrum and $R$-energy of some resistance regular graphs are computed.
\section{Declarations}
 On behalf of all authors, the corresponding author states that there is no conflict of interest.
\bibliography{resreg}
 \bibliographystyle{plain}

\end{document}